\documentclass[]{amsart}
\usepackage{amsthm,wrapfig,amsmath,amssymb,amsfonts,setspace,verbatim,graphicx,mathtools,mathrsfs,commath,float,mdframed,frame,xcolor,qtree,enumitem, ulem}
\usepackage[hidelinks]{hyperref}
\usepackage[T1]{fontenc}

\usepackage[margin=1.57in]{geometry}

\DeclareMathOperator{\dom}{dom}

\DeclareMathOperator{\ran}{ran}
\DeclareMathOperator{\id}{id}

\newtheorem{ut}{Theorem}
\numberwithin{ut}{section}
\newtheorem{up}[ut]{Proposition}
\newtheorem{ul}[ut]{Lemma}

\newtheorem{uc}[ut]{Corollary}
\newtheorem{ucl}[ut]{Claim}

\newtheorem{uq}{Question}

\theoremstyle{remark}
\newtheorem{ur}{Remark}

\theoremstyle{definition}

\begin{document}

\title[Cohesive almost zero-dimensional spaces]{On cohesive almost zero-dimensional spaces}


\author{Jan J. Dijkstra}
\address{PO Box 1180, Crested Butte, CO 81224, United States of America}
\email{jan.dijkstra1@gmail.com}
\author{David S. Lipham}
\address{Department of Mathematics, Auburn University at Montgomery, Montgomery 
AL 36117, United States of America}
\email{dsl0003@auburn.edu, dlipham@aum.edu}

\begin{abstract}
We investigate C-sets in almost zero-dimensional spaces, showing that closed $\sigma$C-sets are C-sets. As corollaries, we  prove that every rim-$\sigma$-compact almost zero-dimensional space is zero-dimensional and that each cohesive almost zero-dimensional space is nowhere rational. To show these results  are sharp, we construct a rim-discrete connected set with an explosion point. We also show  every cohesive almost zero-dimensional subspace of $($Cantor set$)$$\times\mathbb R$ is nowhere dense. 
\end{abstract}

\maketitle

\section{Introduction}

All spaces under consideration are separable and metrizable.

A subset $A$ of a topological space $X$ is called a \textit{C-set in}  $X$ if $A$ can be written as an intersection of clopen subsets of $X$. A {\it$\sigma$C-set}\/ is a countable union of C-sets.
A space $X$ is said to be \textit{almost zero-dimensional} provided every point $x\in X$ has a neighborhood basis consisting of C-sets in $X$. 

A space $X$  is \textit{cohesive} if every point $x\in X$ has a neighborhood which contains no non-empty  clopen subset of $X$. Clearly every  cohesive space is nowhere zero-dimensional. The converse is false, even for   almost zero-dimensional spaces \cite{dij}. Spaces which are both almost zero-dimensional and cohesive include: \begin{align*}
 &\textit{Erd\H{o}s space }  &&\mathfrak E=\{x\in \ell^2:x_i\in\mathbb Q\text{ for each }i<\omega\}\text{ and}\\
 &\textit{complete Erd\H{o}s space } &&\mathfrak E_{\mathrm c}=\{x\in \ell^2:x_i\in \{0\}\cup \{1/n:n=1,2,3,...\}\text{ for each }i<\omega\},
 \end{align*}
 where $\ell^2$ stands for the Hilbert space of square summable sequences of real numbers. Other examples include the homeomorphism groups of the Sierpi\'nski  carpet and Menger universal curve \cite{ov2, dij4}, and various  endpoint sets in complex dynamics \cite{rem,lip}.

Almost zero-dimensionality of $\mathfrak E$ and $\mathfrak E_{\mathrm c}$ follows from the fact that each closed  $\varepsilon$-ball in either space is closed in the  zero-dimensional topology inherited from  $\mathbb Q^\omega$,  which is weaker than the $\ell^2$-norm topology.    The spaces are cohesive because  all non-empty  clopen subsets of $\mathfrak E$  and $\mathfrak E_{\mathrm c}$ are unbounded in the $\ell^2$-norm as proved by Erd\H os \cite{dim}.  Thus, if we add a point $\infty $ to
$\ell^2$ whose neighborhoods are the complements of bounded sets then we have that $\mathfrak E\cup \{\infty\}$ and $\mathfrak E_{\mathrm c}\cup \{\infty\}$ are connected.
The following result is Proposition 5.4 in Dijkstra and van Mill \cite{erd}.

\begin{up}\label{onepoint}Every almost zero-dimensional cohesive space has a one-point connectification. If a space has a one-point connectification then it is cohesive.\end{up}

\noindent Actually, open subsets of non-singleton connected spaces are cohesive, because cohesion is open hereditary \cite[Remark 5.2]{erd}. More information on cohesion and one-point connectifications can be found in \cite{ma}.

In Section 3 of this paper we will show that every cohesive almost zero-dimensional space $E$ is homeomorphic to a dense subset $E'\subset\mathfrak E_{\mathrm c}$ such that  $E'\cup \{\infty\}$ is connected. The result is largely  a consequence of earlier work by Dijkstra and  van Mill  \cite[Chapters 4 and 5]{erd}. We apply the embedding to show that every cohesive almost zero-dimensional subspace of $($Cantor set$)$$\times\mathbb R$ is nowhere dense, and there is a continuous one-to-one image of complete Erd\H{o}s space that is totally disconnected but not almost zero-dimensional. 

In Section 4 of the paper we examine C-sets and the rim-type of almost zero-dimensional spaces. We say  $X$ is \textit{rational at}  $x\in X$ if $x$ has a neighborhood basis of open sets with countable boundaries.   In  \cite[\S 6 Example p.596]{nis}, Nishiura and  Tymchatyn  implicitly proved that  $D^{\mathrm e}$, the set of endpoints of  Lelek's fan  \cite[\S 9]{lel},  is not rational at any of its points.  Results in \cite{bul,cha, 31} later established $D^{\mathrm e}\simeq \mathfrak E_{\mathrm c}$, so $\mathfrak E_{\mathrm c}$ is nowhere rational.   Working in $\ell^2$, Banakh \cite{ban1} recently demonstrated that each bounded open subset of  $\mathfrak E$  has an uncountable boundary. 
We generalize these results   by proving that each cohesive almost zero-dimensional space is nowhere rational. Moreover, every rim-$\sigma$-compact almost zero-dimensional space is zero-dimensional. 
We also find that in almost zero-dimensional spaces cohesion is preserved if we delete $\sigma$-compacta. These results follow from Theorem~\ref{big}, which states that closed  $\sigma$C-sets in almost zero-dimensional spaces are C-sets.
  
  In Section 5 we will construct a rim-discrete connected space $\tau$ with an explosion point.  The  example is partially based on \cite[Example 1]{lip2}, which was constructed by the second author to answer a question from the Houston Problem Book \cite{cookproblem}.  The pulverized complement of the explosion point will be a rim-discrete totally disconnected set which is not zero-dimensional, in contrast with Section 4 results.  Additionally, the rim-discrete property  guarantees the entire connected set has a rational compactification \cite{tym,ili,note}. We therefore solve \cite[Problem 79]{cookproblem} in the context of explosion point spaces. Results from Section 4 indicate that this new solution is optimal. 

In general,   ZD$\implies$AZD$\implies$TD$\implies$HD, where we used abbreviations for zero-dimensional, almost zero-dimensional, totally disconnected, and 
hereditarily disconnected.  
 In certain contexts, these implications can be reversed. For example,  \begin{center}HD$\stackrel{(1)\;}{\implies}$TD$\stackrel{(2)\;}{\implies}$AZD$\stackrel{(3)\;}{\implies}$ZD\end{center} for subsets of hereditarily locally connected continua   \cite[\S50 IV Theorem 9]{kur}. As mentioned above, the implication (3) is valid in the larger class of subsets of rational continua. But \cite[Example 1]{lip2} and the example $\tau$ in Section 5 show that (1) and  (2) are generally false in that context. 

\section{Preliminaries}

A space $X$ is \textit{hereditarily disconnected} if every connected subset of $X$ contains at most one point.  A space $X$ is \textit{totally disconnected} if every singleton in $X$ is  a C-set. A point $x$ in a connected space $X$ is:
\begin{itemize}
\item a  \textit{dispersion point} if $X\setminus \{x\}$ is hereditarily disconnected;
\item an \textit{explosion point} if $X\setminus \{x\}$ is totally disconnected.  
\end{itemize}

If P is a topological property then a space X is \textit{rim-P} provided $X$ has a basis  of open sets whose boundaries have the property P.   \textit{Rational} $\equiv$ rim-countable. \textit{Zero-dimensional} $\equiv$ rim-empty.  
 
Throughout the paper, $\mathfrak C$ will denote the middle-third Cantor set in $[0,1]$.  The coordinate projections in $\mathbb R ^2$ are denoted $\pi_0$ and $\pi_1$;  $\pi_0(\langle x,y\rangle)=x$ and $\pi_1(\langle x,y\rangle)=y$. We define $ \nabla:[0,1]^2\to [0,1]^2$  by  $\langle x,y\textstyle\rangle\mapsto\langle xy+\frac{1}{2}(1-y),y\rangle.$ The image of $\nabla$ is the region enclosed by  the triangle with vertices  $\langle 0,1\rangle$, $\langle \frac{1}{2},0\rangle$, and $\langle 1,1\rangle$.  Note that  $\nabla\restriction [0,1]\times(0,1]$ is a homeomorphism and   $\textstyle \nabla^{-1}(\langle \frac{1}{2},0\rangle)=[0,1]\times \{0\}.$   For each $X\subset \mathfrak C\times (0,1]$  we put  $$\textbf{\d{$\nabla$}}X=\textstyle \nabla(X)\cup\textstyle \{\langle \frac{1}{2},0\rangle\}.$$ The \textit{Cantor fan} is the set  $\nabla(\mathfrak C\times [0,1])=\textbf{\d{$\nabla$}}(\mathfrak C\times (0,1])$, see Figure 1.

Given $X\subset \mathfrak C$, a function  $\varphi:X\to [0,1]$  is     \textit{upper semi-continuous} (abbreviated \textit{USC}) if $\varphi^{-1}[0,t)$ is open in $X$ for every $t\in [0,1]$. Define 
\begin{align*}
G^\varphi_0&=\{\langle x,\varphi(x)\rangle:\varphi(x)>0\}; \text{ and }\\
L^\varphi_0&=\{\langle x,t\rangle:0\leq t\leq \varphi(x)\}.
\end{align*}
We say  $\varphi$ is a \textit{Lelek function} if $\varphi$ is USC and $G^\varphi_0$ is dense in $L^{\varphi}_{0}$. Lelek functions with domain $\mathfrak C$ exist, and if $\varphi$ is a Lelek function with domain $\mathfrak C$ then $\nabla L^\varphi_0$ is a \textit{Lelek fan}, see Figure 2. For example, let $\|\;\;\|$ be the $\ell^2$-norm and identify $\mathfrak C$ with the Cantor set $(\{0\}\cup \{1/n:n=1,2,3,...\})^\omega$. Define  
$\eta(x)=1/(1+\|x\|)$, where $1/\infty=0$. Then  $\mathfrak E_{\mathrm c}$ is homeomorphic to $G^\eta_0$,
$\eta:\mathfrak C\to [0,1]$ is a Lelek function,  and $\nabla L^\eta_0$ is a Lelek fan;  see  \cite{rob} and the proof of \cite[Theorem 3]{crit}.  
 
\begin{figure}[h]
	\begin{minipage}{0.49\textwidth}
       		 \centering   		 \includegraphics[height=40mm,width=49mm]{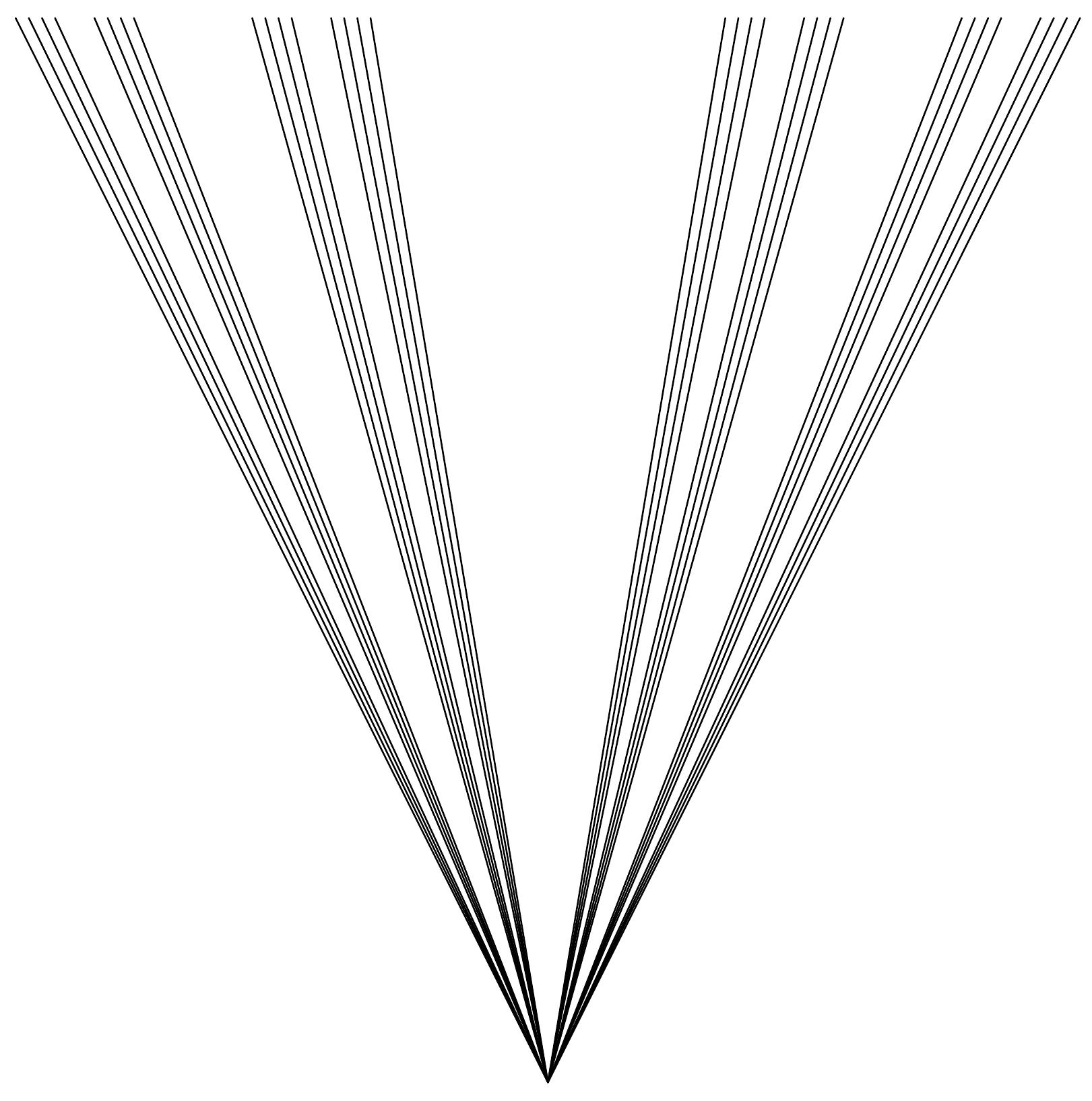}
     		 \caption[]{Cantor fan}
       		 \label{fig:prob1_6_1}
  	\end{minipage} 
	\begin{minipage}{0.49\textwidth}
    		\centering
      		\includegraphics[height=40mm,width=49mm]{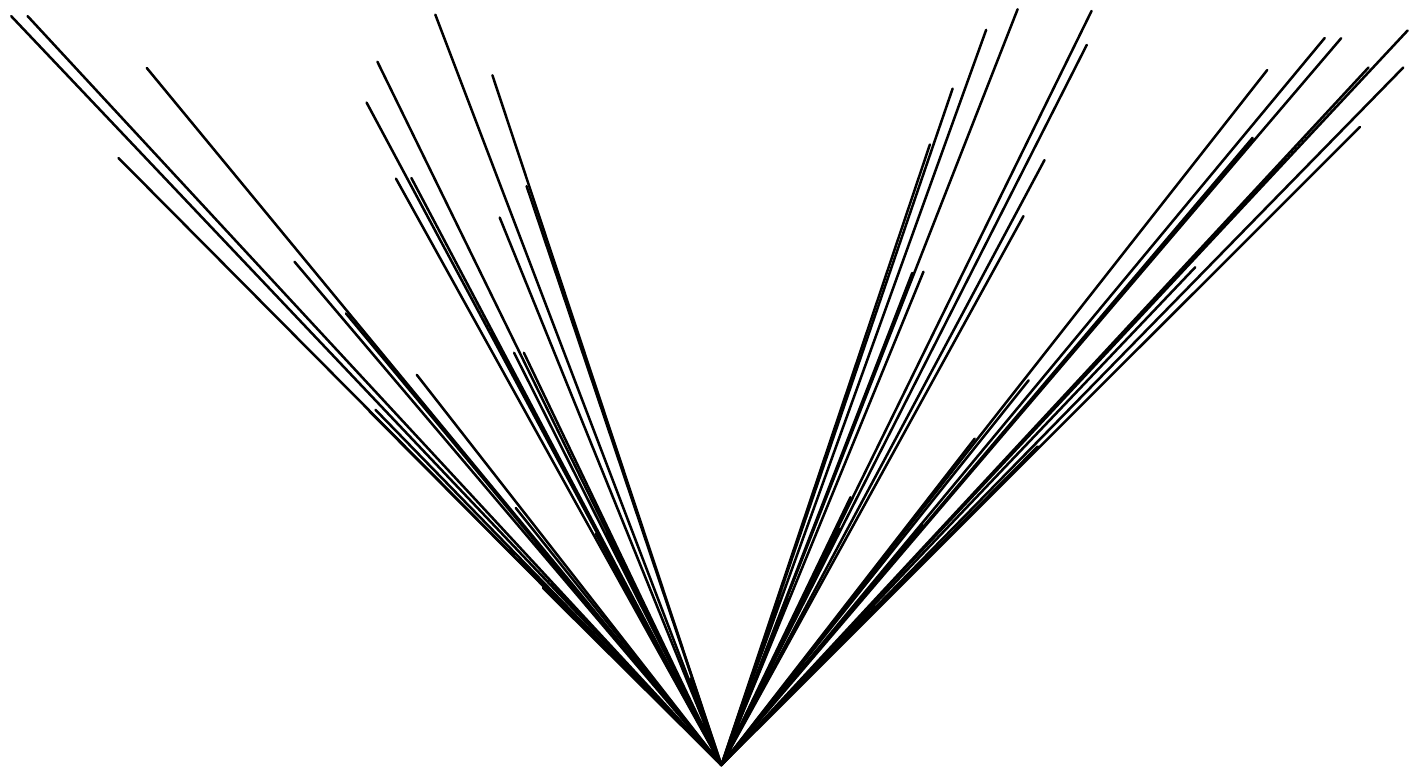}  
   		\caption{Lelek fan}
   	\end{minipage}
\end{figure}

\section{Embedding into fans and complete Erd\H{o}s space}

Let $E$  be any non-empty cohesive almost zero-dimensional space. Dijkstra and van Mill proved:  \textit{There is a Lelek function $\chi:X\to[0,1)$ such that $E$ is homeomorphic to $G^\chi_0$ and hence $E$ admits a dense embedding in $\mathfrak E_{\mathrm c}$}  \cite[Proposition 5.10]{erd}. 
We observe:

\begin{ut}\label{emb}For the  Lelek function $\chi$ constructed  in \cite{erd}, $\textnormal{\textbf{\d{$\nabla$}}} G^\chi_0$ is connected. Thus, there is a dense homeomorphic embedding $\alpha:E\hookrightarrow \mathfrak E_{\mathrm c}$ such that $\alpha(E)\cup \{\infty\}$ is connected.  
\end{ut}

\begin{proof} In \cite{erd}, $\chi$ is constructed via  two USC functions $\varphi$ and $\psi$ which have the same zero-dimensional domain $X$.  First, $\varphi$ is given by   \cite[Lemma 4.11]{erd} such that
$E$ is homeomorphic to $G^\varphi_0$.  And then,  in the proof of \cite[Lemma 5.8]{erd}, $\psi$ is defined by $$\psi(x)=\lim\limits_{\varepsilon\to 0^+}\inf J_\varepsilon (x)\text{, where}$$ 
\begin{align*}
U_\varepsilon(x)&=\{y\in X:d(x,y)<\varepsilon)\text{; and }\\
J_\varepsilon(x)&=\{t\in [0,1):U_\varepsilon(x)\times (t,1)\cap G^\varphi_0\text{ contains no non-empty clopen subset of }G^\varphi_0\}.\end{align*}
Notice  that $J_\varepsilon(x)$ becomes larger as $\varepsilon$ decreases,  so its infimum decreases.  Thus $\psi(x)$ is well-defined.  Finally,   $\chi$  is defined so that $\langle x,\varphi(x)\rangle \mapsto \langle x,\chi(x)\rangle$  is a homeomorphism and   $\chi\leq \varphi-\psi$ \cite[Lemma 4.9]{erd}.
 
To prove that $\textnormal{\textbf{\d{$\nabla$}}} G^\chi_0$ is connected, we let $A$ be any non-empty clopen subset of $G^\chi_0$ and show $0\in \overline{\pi_1(A)}$.  Define $y=\inf\{\varphi(x): x\in \pi_0(A)\}$ and let $\varepsilon>0$. Pick an $x\in \pi_0(A)$ with $\varphi(x)<y+\varepsilon$. 
Since $\{\langle x,\varphi(x)\rangle:x\in \pi_0(A)\}$ is a clopen subset of $G^\varphi_0$ and $X$ is zero-dimensional, $\psi(x)\geq y$. We have $\langle x,\chi(x)\rangle\in A$ and 
$\pi_1(\langle x,\chi(x)\rangle)=\chi(x)\leq  \varphi(x)-\psi(x)<(y+\varepsilon)-y=\varepsilon.$ 
Since $\varepsilon$ was an arbitrary positive number, this shows that $0\in \overline{\pi_1(A)}$.

We will now construct $\alpha$.  Since $\chi$ is Lelek, $X$ is perfect, so we may assume  $X$ is dense in $\mathfrak C$. Now $\chi$ extends to a Lelek function $ \overline{\chi}:\mathfrak C\to [0,1]$ such that $G^\chi_0$ is dense in $G^{\overline{\chi}}_0$ \cite[Lemma 4.8]{erd}.  In particular,  $\nabla L^{\overline{\chi}}_0$  is a Lelek fan.  By \cite{bul,cha} the Lelek fan is unique,  so there is a homeomorphism $\Xi:\nabla L^{\overline{\chi}}_0\to \nabla L^\eta_0$ (recall   $\eta$ from Section 2).  We observe that $\Xi(\textstyle\textbf{\d{$\nabla$}}G^{\overline{\chi}}_0)=\textstyle\textbf{\d{$\nabla$}}G^\eta_0\simeq \mathfrak E_{\mathrm c}\cup \{\infty\}$. So there is a homeomorphism $\gamma:\textstyle\textbf{\d{$\nabla$}}G^{\overline{\chi}}_0\to \mathfrak E_{\mathrm c}\cup \{\infty\}$.  We know there is also a homeomorphism $\beta:E\to \textstyle\textbf{{$\nabla$}}G^\chi_0$. Let $\alpha=\gamma\circ\beta$. Notice that $\alpha(E)\cup \{\infty\}=\gamma(\textstyle\textbf{\d{$\nabla$}}G^{{\chi}}_0)$ is connected.
\end{proof}

\begin{uc}\label{som}If $Y$ is a complete space containing $E$, then there is a complete cohesive almost zero-dimensional space $E'$ such that $E\subset E'\subset Y$.\end{uc}

\begin{proof}Let $\alpha:E\hookrightarrow  \mathfrak E_{\mathrm c}$ be given by Theorem \ref{emb}.   Since $Y$ and $\mathfrak E_{\mathrm c}$ are both complete,   Lavrentiev's Theorem \cite[Theorem 4.3.21]{eng} says $\alpha$ extends  to a  homeomorphism  between $G_\delta$-sets $E'$ and $A$ such that   $E\subset E'\subset Y$ and $\alpha(E)\subset A\subset \mathfrak E_{\mathrm c}$.  Since $\alpha(E)$ is dense in $\mathfrak E_{\mathrm c}$ and $\alpha(E)\cup \{\infty\}$ is connected, $A\cup \{\infty\}$ is  connected. So $E'$ is cohesive. \end{proof}



\begin{ut}\label{nw}Every cohesive almost zero-dimensional subset of  $\mathfrak C\times\mathbb R$ is nowhere dense.\end{ut}
    
\begin{proof}Cohesion is open-hereditary \cite[Remark 5.2]{erd}. By  self-similarity of $\mathfrak C\times\mathbb R$,   it therefore suffices to show  there is no dense cohesive almost zero-dimensional subspace of $\mathfrak C\times\mathbb R$. Suppose on the contrary that $E$ is such a space.  By Corollary \ref{som}  there is a complete cohesive almost zero-dimensional $X\subset \mathfrak C\times\mathbb R$  such that $E\subset X$.  Then $X$ is a dense  $G_\delta$-subset of $\mathfrak C\times\mathbb R$, so by \cite{brouwer, KU} there exists $c\in \mathfrak C$ such that $\overline{X\cap(\{c\}\times \mathbb R)}=\{c\}\times\mathbb R$.  Let $x=\langle c,r\rangle\in X$. We obtain a contradiction by showing $X$ is zero-dimensional at $x$.    Let $V\times (a,b)$ be any regular open subset of  $\mathfrak C\times \mathbb R$ which contains $x$. There exist an $r_1\in (a,r)$ and an $r_2\in (r,b)$ such that $x_1=\langle c,r_1\rangle$ and $x_2= \langle c,r_2\rangle$ are in $ X$. Since $X$ is totally disconnected, there are $X$-clopen sets $W_1$ and $W_2$ such that $x_1\in W_1$, $x_2\in W_2$, and $x\notin W_1\cup W_2$. Let $U_1,U_2\subset V$ be $\mathfrak C$-clopen sets such that  $x_i\in (U_i\times \{r_i\})\cap X\subset W_i$ for each $i\in\{1,2\}$.  Then $[(U_1\cap U_2)\times[r_1,r_2]\setminus (W_1\cup W_2)]\cap X$ is an  $X$-clopen subset of $V\times (a,b)$ which contains $x$.  This shows that $X$ is zero-dimensional at $x$.  \end{proof}


Corollary \ref{nw} shows that a certain  continuous one-to-one  image of $\mathfrak E_{\mathrm c}$  is totally disconnected but not almost zero-dimensional.  Define $$f:\mathfrak E_{\mathrm c}\to (\{0\}\cup \{1/n:n=1,2,3,...\})^\omega\times [0,1]$$ by $\textstyle f(x)=\big\langle x,\frac{1+\sin\|x\|}{2}\big\rangle.$    Let $Y=f(\mathfrak E_{\mathrm c})$. 
Clearly,  $f$ is one-to-one and continuous, and $Y$ is totally disconnected. The example $Y$ is essentially the same as \cite[Example $X_2$]{lip0} and therefore by \cite[Propositions 3 and 5]{lip0} $Y$ is dense in $\mathfrak C\times[0,1]$ and $\nabla Y$ is connected.  Thus $Y$ is cohesive. By Theorem~\ref{nw} $Y$ is not almost zero-dimensional. 
Both this example and the space $\tau$ constructed in Section 5 show that Theorem~\ref{nw} does not extend to totally disconnected spaces.

\section{$\sigma$C-sets and rim-type}
\begin{ur}\label{rem}
If $x\in A\mathrm{o}\subset X$ with $\partial A$ a C-set in $X$ then there is a clopen set $C$ with $x\in C$ and $C\cap \partial A=\varnothing$ and hence $C\cap A\mathrm{o}=C\cap\overline A$ is also clopen. Consequently, rim-C is equivalent to zero-dimensional.
\end{ur}

\begin{ul}\label{se}For every two disjoint C-sets in a space, there is a clopen set containing one and missing the other.
\end{ul}

\begin{proof}Identical to the proof of \cite[Lemma 1.2.6]{eng1}. \end{proof}

\begin{ut}Let $A$ be a subset of an almost zero-dimensional space $X$. If there is a $\sigma$C-set  $B$ with $\partial A\subset B\subset \overline A$,  then $\overline A$ is a C-set.\label{lab}  \end{ut}

\begin{proof}Suppose $B=\bigcup \{B_i:i<\omega\}$  where each $B_i$ is a C-set, and $\partial A\subset B\subset \overline A$. To prove $\overline A$ is a C-set, it suffices to show that for every  $x\in X\setminus \overline A$ there is an  $X$-clopen set $C$ such that $x\in C\subset X\setminus \overline A$. 

Let $x\in X\setminus\overline A$. By the Lindel\"{o}f property and almost zero-dimensionality,  it is possible to write the open set $X\setminus \overline A$ as the union of countably many C-sets in $X$ whose interiors cover $X\setminus \overline A$.  The property of being a C-set is closed under finite unions, so  there is an increasing sequence of C-sets $D_0\subset D_1\subset\dotsb$ with  $x\in D_0$ and 
$$\textstyle\bigcup \{D_i:i<\omega\}= \bigcup \{ D_i ^\mathrm{o}:i<\omega\}=X\setminus \overline A.$$  
By Lemma \ref{se}, for each $i<\omega$ there is an $X$-clopen set  $C_i$ such that $D_i\subset C_i\subset X\setminus B_i.$  Let $\textstyle C=\bigcap  \{C_i:i<\omega\}\setminus A^\mathrm{o}$. Clearly, $C$ is closed, $x\in C$, and $$C\subset X\setminus (A^\mathrm{o}\cup B) = X\setminus \overline A.$$   Further, if $y\in C$ then there exists $j<\omega$  such that $y\in D_j^\mathrm{o}$. The open set  $D_j^\mathrm{o}\cap \bigcap\{C_i:i<j\}$ witnesses that $y\in C^\mathrm{o}$.  This  shows $C$ is open and thus clopen.\end{proof}
\begin{ut}\label{big} In an almost zero-dimensional space, every closed $\sigma$C-set is a C-set.  \end{ut}
\begin{proof}Given a closed $\sigma$C-set $A$, apply  Theorem \ref{lab} with $B=A$.\end{proof}
With Remark~\ref{rem} we get:
\begin{uc}\label{sig}Every rim-$\sigma$C almost zero-dimensional space is zero-dimensional.\end{uc}

Since compacta are C-sets in totally disconnected spaces we also have:
\begin{uc}\label{sigb}Every almost zero-dimensional space that is rim-$\sigma$-compact or rational is zero-dimensional.\end{uc}

A space is called {\it nowhere rim-$\sigma$C\/} ({\it nowhere rim-$\sigma$-compact\/} respectively {\it nowhere rational\/}) if no point has a neighborhood basis consisting of sets that have boundaries that are
$\sigma$C-sets ($\sigma$-compact respectively countable).  With Theorem~\ref{big} and Remark~\ref{rem} we also find:
\begin{uc}\label{c6}Cohesive almost zero-dimensional spaces are nowhere rim-$\sigma$C and hence nowhere rim-$\sigma$-compact and nowhere rational.  \end{uc}

Thus there are no rim-$\sigma$-compact or rational connected spaces $Y$ with a point $p$ such that $Y\setminus\{p\}$ is almost zero-dimensional, using Proposition~\ref{onepoint}.

\begin{ut} \label{c7} If $X$ almost zero-dimensional,  $Y=X\cup\{p\}$ is connected, and $K\subset X$ is $\sigma$-compact,  then  $Y\setminus K$ is connected.  
\end{ut}

\begin{proof}Suppose $X$ almost zero-dimensional,  $Y$ is connected, and $K\subset X$ is $\sigma$-compact.  Striving for a contradiction suppose $Y\setminus K$ is not connected. Then $Y\setminus K$ is the union of two non-empty relatively closed subsets $A$ and $B$ such that $A\cap B=\varnothing$. We may assume that $p\in B$.  The closures of $A$ and $B$ in the open set $Y\setminus(\overline A\cap\overline B)$ are disjoint, so they are contained in disjoint $Y$-open sets $U$ and $V$. Note that  $\partial U$ in $Y$ is contained in $K$ and is therefore $\sigma$-compact and hence a $\sigma$C-set in the totally disconnected space $X$. By Theorem \ref{big} $\partial A$ is a C-set in $X$. So by Remark~\ref{rem} $U$ contains a nonempty clopen subset $C$ of $X$.  Note that $X$ is open in $Y$ and $U$ is contained in the $Y$-closed
set $Y\setminus B$ so $C$ is also clopen in $Y$. This violates the assumption that $Y$ is connected.  \end{proof}
Since $\mathfrak E\cup\{\infty\}$ and $\mathfrak E_{\mathrm c}\cup\{\infty\}$  are connected we have:
\begin{uc}Bounded neighborhoods in $\mathfrak E$ and $\mathfrak E_{\mathrm c}$  do not have  $\sigma$-compact boundaries.\end{uc}
 Combining Theorem~\ref{c7} with Proposition~\ref{onepoint} we find:
\begin{ut}If $X$ is cohesive and almost zero-dimensional and $K\subset X$ is $\sigma$-compact, then $X\setminus K$ is cohesive. \end{ut}

  
For the spaces $\mathfrak E$, $\mathfrak E_{\mathrm c}$, and $\mathfrak E_{\mathrm c} ^\omega$ there is a stronger result: in these spaces $\sigma$-compacta are negligible,
see \cite{negl}, \cite{31}, and \cite{stable}.

A connected space $X$ is \textit{$\sigma$-connected} if $X$ cannot be written as the union of $\omega$-many pairwise disjoint non-empty closed subsets.  Note that the Sierpi\'nski Theorem   \cite[Theorem 6.1.27]{eng} states that every continuum is $\sigma$-connected. Lelek \cite[P4]{lel0} 
asked whether every connected space with a dispersion point is $\sigma$-connected.
  Duda   \cite[Example 5]{dud}  
answered this question in the negative.

\begin{ut}If a space $X$ contains an open  almost zero-dimensional subspace $O$ with $O\not=\varnothing$ and $X\setminus O\not=\varnothing$ then $X$ is not $\sigma$-connected.\end{ut}

\begin{proof}We may assume that $X$ is connected. Since $O$ is almost zero-dimensional and open we can find for every $x\in O$ a C-set $A_x$ in $O$ that is closed in $X$
and with $x\in A_x\mathrm{o}$. Select a countable subcovering $\{B_i:i<\omega\}$ of $\{A_x:x\in O\}$. Since the union of two C-sets is a C-set we can arrange that 
$B_i\subset B_{i+1}$ for each $i<\omega$. Also we may assume that $B_0=\varnothing$. Since $B_i$ is a C-set in $O$ we can find an $O$-clopen covering $\mathcal C_i$
of $O\setminus B_i$. We  may assume that $\mathcal C_i=\{C_{ij}:j<\omega\}$ is countable. Moreover, by clopenness we can arrange that $\mathcal C_i$ is a disjoint collection. Consider the countable closed disjoint covering $$\mathcal F=(\{X\setminus O\}\cup \{C_{ij}\cap B_{i+1}:i,j<\omega\})\setminus\{\varnothing\}$$ of $X$. If $\mathcal F$ is finite then 
$O$ is closed and hence clopen, violating the connectedness of $X$. Thus $X$ is not $\sigma$-connected.
\end{proof}
Since every cohesive almost zero-dimensional space has a one-point connectification by  Proposition~\ref{onepoint}
 it produces an example in answer to Lelek's question. These examples are explosion point spaces rather than just dispersion point spaces. In particular,
 we have that $\mathfrak E\cup \{\infty\}$ and $\mathfrak E_{\mathrm c}\cup \{\infty\}$ are counterexamples. Note that  $\mathfrak E_{\mathrm c}\cup \{\infty\}$ is complete
 which is optimal because $\sigma$-compact dispersion point spaces cannot exist.

\section{A rim-discrete space with an explosion point} 

Let $\mathfrak C$, $\nabla$ and  \textbf{\d{$\nabla$}} be as defined in Section 2.  We will construct a function $\tau:P \to(0,1)$ with domain $P\subset \mathfrak C$ such that:

\begin{itemize}
\item[(1)] $\tau$ is a dense subset of $\mathfrak C\times(0,1)$;
\item[(2)] $\textbf{\d{$\nabla$}}\tau$ is connected; and
\item[(3)] $\textbf{\d{$\nabla$}}\tau$ is rim-discrete. 
\end{itemize}

\noindent Here we identify a function like $\tau$ with its graph (in the product topology).  Clearly $\tau$ will be totally disconnected.  Note that $\tau$ cannot be almost zero-dimensional by 
(2), (3), and Corollary~\ref{sigb}  or (1), (2), and Theorem \ref{nw}.

\subsection{Construction of $Z$}We first describe a rim-discrete connectible set $Z\subset \mathfrak C\times\mathbb R$ similar to $Y$ in \cite[Example 1]{lip2}.

Let $E$ be the set of endpoints of connected components of $\mathbb R \setminus \mathfrak C$.  For each $\sigma\in 2^{<\omega}$, let $n=\dom(\sigma)$ and define $$B(\sigma)=\big[\textstyle\sum_{k=0}^{n-1} \frac{2\sigma(k)}{3^{k+1}},  \textstyle\sum_{k=0}^{n-1} \frac{2\sigma(k)}{3^{k+1}}+\frac{1}{3^n}\big]\cap \mathfrak C.$$ Here, $B(\varnothing)=[0,1]\cap \mathfrak C=\mathfrak C$. The set of all $B(\sigma)$'s is the canonical clopen basis for $\mathfrak C$.

Suppose $\sigma\in 2^{<\omega}$, $Q$ is a countable dense subset of $ B(\sigma)\setminus E$, and $a$ and $b$ are real numbers with $a<b$.  Fix an enumeration $\{q_m:m<\omega\}$ for $Q$, and define a function   $$f=f_{\langle Q,\sigma,a,b\rangle}:B(\sigma)\to [a,b]$$ by the formula $$\textstyle f(c)=a+(b-a)\cdot \sum\{2^{-m-1}:m<\omega\text{ and } q_m<c\}.$$  Note that:
\begin{itemize}\renewcommand{\labelitemi}{\scalebox{.5}{{$\blacksquare$}}}
\item  $f$ is  well-defined  and non-decreasing;
\item $f\restriction B(\sigma)\setminus E$ is one-to-one;
\item   $f$ has the same value at consecutive elements of $E$;
\item  $Q$ is the set of discontinuities of $f$;  and
\item the discontinuity at $q_m$ is caused by a jump of height $(b-a)\cdot2^{-m-1}$. 
\end{itemize}
Let  $$D=\textstyle D_{\langle Q,\sigma,a,b\rangle}=f\cup \bigcup \{\{q_m\}\times [f(q_m),f(q_m)+(b-a)\cdot2^{-m-1}]:m<\omega\}.$$  Thus $D$ is equal to (the graph of) $f$ together with vertical arcs corresponding to the jumps in  $f$. Note that $\pi_1(D)=[a,b]$ and  $D$ is compact.

Let $\{Q^n_i:n,i<\omega\}$ be a collection of pairwise disjoint countable dense subsets of $\mathfrak C\setminus E$. 
As in \cite[Example 1]{lip2}, it is possible to recursively  define  a sequence $\mathcal R_0,\mathcal R_1, ...$  of finite partial tilings of $\mathfrak C\times \mathbb R$ so that for each $n<\omega$:

\begin{itemize} \item[i.]    $\mathcal R_n$ consists of rectangles $R^n_i=B(\sigma^n_i)\times [a^n_i,b^n_i]$, where $i<|\mathcal R_n|<\omega$, $\sigma^n_i\in 2^n$,   and $0<b^n_i-a^n_i\leq\frac{1}{n+1}$ for all $i<|\mathcal R_n|$;

\item[ii.] the  sets $$D^n_i=D_{\langle Q^n_i\cap B(\sigma^n_i),\sigma^n_i,a^n_i,b^n_i\rangle}$$ are such that $D^n_i\cap D^k_j=\varnothing$ whenever $k< n$ or $i\neq j$;

\item[iii.] for every arc  $I\subset \mathfrak C\times[-n,n+1]\setminus \bigcup \{D^k_i:k\leq n\text{ and }i<|\mathcal R_k|\}$    there are  integers $i<|\mathcal R_n|$,  $k\leq n$, and $j<|\mathcal R_{k}|$ such that $I\subset R^n_{i}\cup R^{k}_{j}$ and $d(I,D^k_j)\leq\frac{1}{3^n}$, where $d$ is the standard metric on $\mathbb R^2$.
\end{itemize}
Let $M^n_i$ be the (discrete) set of midpoints of the vertical arcs in $D^n_i$. The key difference between the sets $M^n_i$ and  the $T^n_i(M)$ defined in  \cite[Example 1]{lip2} is that here we have guaranteed $\pi_0(M^n_i)\cap \pi_0(M^k_j)\subset Q^n_i\cap Q^k_j=\varnothing$ whenever $n\neq k$ or $i\neq j$, whereas a vertical line could intersect multiple $T^n_i(M)$'s.  

Let $\{D_n:n<\omega\}$ and $\{M_n:n<\omega\}$ be the sets of all $D^n_i$'s and $M^n_i$'s, respectively. Properties (i) through (iii) guarantee the set $Z=\mathfrak C\times \mathbb R \setminus\bigcup \{D_n\setminus M_n:n<\omega\}$ is rim-discrete; see  \cite[Claims 1 and 3]{lip2}. Essentially, $\tau$ will be a subset of $Z$ containing all $M_n$'s, but will be vertically compressed from $\mathfrak C\times \mathbb R$ into $\mathfrak C\times(0,1)$.

\subsection{Construction of  $\overline g$}We now construct a connected function $\overline g$  on which $\tau$ will be based. Let $\xi:\mathbb R \to (0,1)$ be a homeomorphism, e.g. $\xi=\frac{1}{2}+ \frac{1}{\pi}\arctan$.  Let $\phi:[0,1]\to[0,1]$ be the Cantor function \cite{dog}, and put $\textstyle\Phi=\phi\times \xi.$ Then each $\Phi(D_n)$ is an arc which  resembles the  graph of $\phi$ reflected across the diagonal $x=y$.  See Figure 3.
\begin{figure}[h]
        \centering
     \includegraphics[scale=0.34]{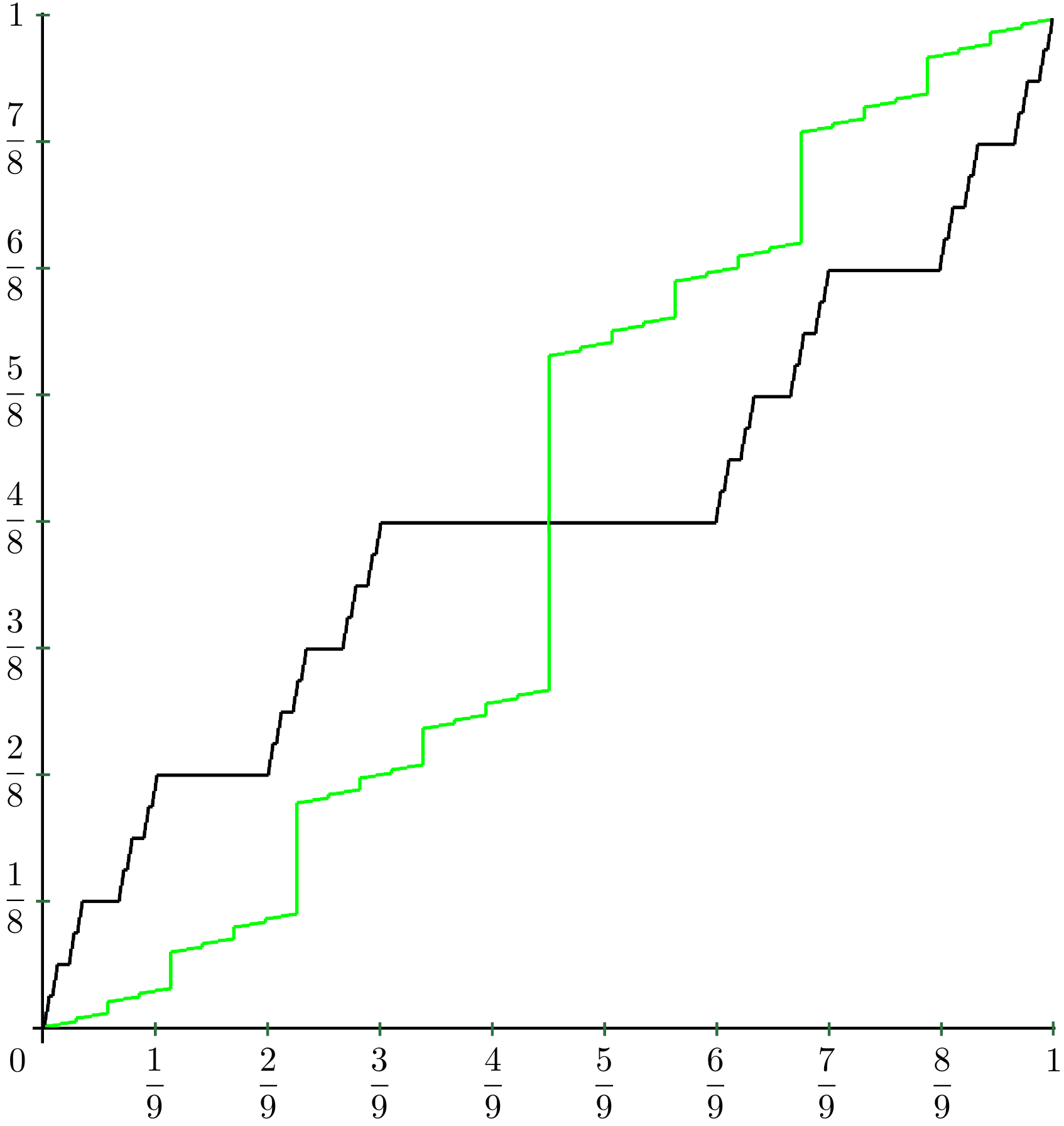}
        \caption{Graph of $\phi$ (black) and its ``inverse'' (green)}
\end{figure}

Note that $\phi(E)$ is the set of dyadic rationals in $[0,1]$. Let $$g=\textstyle(\phi(E)\times\{0\})\cup \bigcup\{\Phi(M_n):n<\omega\}.$$   Since $\pi_0\restriction M_n$ is one-to-one and the $\pi_0(M_n)$'s are pairwise disjoint,  $g$ is a function.  Also, $$\dom(g)=\textstyle\phi(E)\cup \bigcup \{\pi_0(\Phi(M_n)):n<\omega\}$$ is countable and $\ran(g)\subset [0,1)$. Our goal is to extend $g$ to a connected function $\overline g:[0,1]\to (-1,1)$. This will be accomplished with the help of two claims. By a \textit{continuum} we shall mean a compact connected metrizable space with more than one point.  

\begin{ucl} Fix $n<\omega$ and put $D=D_n$ and $M=M_n$.  Let $A\subset [0,1]$ have a dense complement and let $K\subset \Phi(D)\cup (A\times(-1,1))$  be a continuum. If $|\pi_0(K)|>1$ then $K\cap \Phi(M)\neq\varnothing$. \end{ucl}

\begin{proof}Let $a$ and $b$ be two points in $K$ such that $\pi_0(a)<\pi_0(b)$.  Since $\pi_0(K)$ is an interval contained in the union of the zero-dimensional set $A$ and the interval $\pi_0(D)$
we have $\pi_0(K)\subset \pi_0(D)$.  Noting that $\pi_0(M)$ is dense in $\pi_0(D)$ we find  a $p\in\Phi(M)$ such that $\pi_0(p)\in (\pi_0(a),\pi_0(b))$. If $p\notin K$, then we can find $c, d \in [0,1]\setminus A$ such that  $U=[c,d]\times\{\pi_1(p)\}$ is disjoint from $K$ and $\pi_0(a)<c<\pi_0(p)<d<\pi_0(b)$.
Then $U\cup(\{c\}\times (\pi_1(p),1))\cup(\{d\}\times (-1,\pi_1(p))$ separates $K$ with $a$ and $b$ on opposite sides. This contradicts our assumption that $K$ is connected.  Therefore $p\in K$.\end{proof}

\begin{ucl} Let $A\subset [0,1]$ be any countable set, and let $K\subset \bigcup \{\Phi(D_n):n<\omega\}\cup (A\times(-1,1))$  be a continuum.  If $|\pi_0(K)|>1$ then $K\cap \Phi(M_n)\neq\varnothing$ for some $n<\omega$.\end{ucl}

\begin{proof}For each $x\in[0,1]$, let $K_x=K\cap(\{x\}\times(-1,1))$. Let $\mathcal K$ be the  decomposition of $K$ consisting of every connected component of every non-empty $K_x$. Applying \cite[Lemma 6.2.21]{eng} to the perfect map $\pi_0\restriction K$, we see that $\mathcal K$ is upper semi-continuous. If $q:K\to K'$ is the associated (closed) quotient mapping then $K'$ is also a continuum.  Consider the countable covering $\mathcal V$ of $ K'$ consisting of the compacta $q(K_x)$ for $x\in A$ and $q(\Phi(D_n)\cap K)$  for  all $n<\omega$. By the Baire Category Theorem there is an element of $\mathcal V$ that has nonempty interior in $ K'$ and hence contains a (non-degenerate) continuum $ C'$  by \cite[Theorem 6.1.25]{eng}.  Each $q(K_x)$ is zero-dimensional by \cite[Theorem 6.2.24]{eng}, so $ C'\subset q(\Phi(D_n)\cap K)$ for some $n<\omega$. Since $q$ is a closed monotone map, the pre-image $C=q^{-1}( C')$ is a continuum  by \cite[Theorem 6.1.29]{eng}. Note that $|\pi_0(C)|>1$ because otherwise $C'$ would be a subset of some zero-dimensional $q(K_x)$.  If   $x\notin A$ then  each connected component of $K_x$ is contained in a single $\Phi(D_i)$ by the Sierpi\'nski Theorem \cite[Theorem 6.1.27]{eng},
because the $\Phi(D_i)$'s are disjoint. Thus $q(\Phi(D_n)\cap K_x)$ is disjoint from $q(\Phi(D_i)\cap K_x)$ for each $i\not=n$. So $C\subset (A\times(-1,1))\cup\Phi(D_n)$.  By Claim 5.1 we have that $C\cap \Phi(M_n)\neq\varnothing$.\end{proof}

Now let $\{x_\alpha:\alpha<\mathfrak c\}$ enumerate the set $[0,1]\setminus \dom(g)$.  Let $\{K_\gamma:\gamma<\mathfrak c\}$ be the set of continua in $[0,1]\times (-1,1)$ such that:
\begin{itemize}\renewcommand{\labelitemi}{\scalebox{.5}{{$\blacksquare$}}}
\item $K_\gamma$ is not contained in any vertical line; and 
\item $K_\gamma \cap \Phi(M_n)= \varnothing$ for all $n<\omega$.
\end{itemize}
For each $\alpha<\mathfrak c$ let $l_{\alpha}=(\{x_\alpha\}\times(-1,1))\setminus\bigcup \{\Phi(D_n):n<\omega\}$.  By transfinite induction we  define for each $\alpha<\mathfrak c$ an ordinal $$\gamma(\alpha)=\min\{\gamma<\mathfrak c:l_\alpha\cap K_\gamma\neq\varnothing \text{ and } \gamma\neq \gamma(\beta)\text{ for any }\beta<\alpha\}.$$ 
We verify that the one-to one function $\gamma:\mathfrak c\to \mathfrak c$ is well-defined. Let $\alpha<\mathfrak c$ so $x_\alpha\notin \dom(g)$ and $x_\alpha\notin \pi_0(\Phi(M_n))$ for each $n$. Since $M_n$ contains the midpoints of 
all vertical intervals in $D_n$ we have that $\{x_\alpha\}\times(-1,1)$ contains at most one point of $\Phi(D_n)$. Let $A$ be the countable set 
$$\bigcup_{n<\omega}\pi_1((\{x_\alpha\}\times(-1,1))\cap\Phi(D_n))\cup\Phi(M_n)).$$
If $a\in(-1,1)\setminus A$ then $K=[0,1]\times\{a\}$ misses every $\Phi(M_n)$ so $K=K_\beta$ for some $\beta<\mathfrak c$. Also we have $l_\alpha\cap K_\beta\not=\varnothing$. 
Since $|(-1,1)\setminus A|=\mathfrak c$ we have that $\gamma$ is well-defined.

For every $\alpha<\mathfrak c$ choose a $y_\alpha\in \pi_1( l_\alpha\cap K_{\gamma(\alpha)})$. Define $$\overline g=g\cup \{\langle x_\alpha,y_\alpha\rangle:\alpha<\mathfrak c\}$$ and note that $\overline g:[0,1]\to (-1,1)$ is a function.  To prove that the graph of $\overline g$ is connected let $K$ be a continuum in $[0,1]\times(-1,1)$ such that $|\pi_0(K)|>1$. We show that $K\cap\overline g\not=\varnothing$.
The set $K$ intersects some $\Phi(M_n)$, which is a subset of $\overline g$, or $K=K_\alpha$ for some $\alpha<\mathfrak c$.
 By the contraposition of Claim 5.2, the projection $A=\pi_0(K_\alpha\setminus \bigcup \{\Phi(D_n):n<\omega\})$ is uncountable.  $A$ is a continuous image of  a Polish space, so in fact  it has cardinality $\mathfrak c$ by \cite[Corollary 11.20]{jech}. Since  $[0,1]\setminus \{x_\beta:\beta<\mathfrak c\}=\dom(g)$ is countable, this means  $B=\{\beta<\mathfrak c:l_\beta\cap K_\alpha\neq\varnothing\}$ has cardinality $\mathfrak c$. Assuming that $K_\alpha\cap\overline g=\varnothing$ we find that $\alpha$ cannot be in the range of $\overline g$. If $\beta\in B$ then $l_\beta\cap K_\alpha\not=\varnothing$
 so by the definition of $\gamma$  we have $\gamma(\beta)<\alpha$. Thus $\gamma\restriction B$ is a one-to-one function from $B$ into $\{\delta:\delta<\alpha\}$ and we have the desired contradiction.   So (the graph of) $\overline g$  intersects each continuum in $[0,1]\times (-1,1)$ not lying wholly in a vertical line.  By  \cite[Theorem 2]{jon},  $\overline g$ is connected.
 

\subsection{Definition and properties of $\tau$}Observe that $\overline g\circ \phi\subset (\mathfrak C\times(-1,1))\cup([0,1]\times\{0\})$. Let $$\textstyle\tau=(\overline g\circ \phi)\cap (0,1)^2.$$  The domain of $\tau$ is  the set $P=\pi_0(\tau)\subset \mathfrak C$. 

Let $X=\nabla (\overline g \cap ((0,1)\times[0,1)))$.  If $A$ is any clopen subset of  $X$ with $\langle \frac{1}{2},0\rangle\in A$, then $A=X$.  Otherwise, $\nabla^{-1}(X\setminus A)$ would be a non-empty proper clopen subset of $\overline g$, contrary to the fact that $\overline g$ is connected. Therefore $X$ is connected. Note that $\textstyle \textbf{\d{$\nabla$}}\tau\simeq X$, so  $\textstyle \textbf{\d{$\nabla$}}\tau$ is also connected. Finally, let  $\textstyle\Xi=\id_{\mathfrak C}\times \xi.$ By \cite[Claims 3 and 4]{lip2} and the construction of $Z$, $\textbf{\d{$\nabla$}}\Xi(Z)$ is rim-discrete.   We have $\textbf{\d{$\nabla$}}\tau\subset \textbf{\d{$\nabla$}}\Xi(Z)$, so $\textbf{\d{$\nabla$}}\tau$ is rim-discrete.

\subsection{Two questions.}
A continuum is \textit{Suslinian} if it contains no  uncountable collection of pairwise disjoint (non-degenerate)  subcontinua \cite{lel2}. The class of Suslinian continua is slightly larger than the class of rational continua.

\begin{uq}Can $\mathfrak E_{\mathrm c}$ be embedded into a Suslinian continuum? 
\end{uq}

\begin{uq}Can $\mathfrak E_{\mathrm c}$ be densely embedded into the plane $\mathbb R ^2$?  
\end{uq}




\end{document}